\documentclass{siamart1116}

\usepackage{latexsym}
\usepackage{amsmath}
\usepackage{amssymb}
\usepackage{amsfonts}
\usepackage[latin1]{inputenc}
\usepackage{mathrsfs}
\usepackage{amsfonts}
\usepackage{graphicx}
\usepackage{colortbl,dcolumn}
\usepackage{amsmath}
\usepackage{amsfonts}
\usepackage{graphicx}
\usepackage{algorithmic}
\ifpdf
  \DeclareGraphicsExtensions{.eps,.pdf,.png,.jpg}
\else
  \DeclareGraphicsExtensions{.eps}
\fi

\newtheorem{tm}{Theorem}[section]
\newtheorem{rk}{Remark}[section]

\newtheorem{prop}{Proposition}[section]
\newtheorem{lm}{Lemma}[section]

\newtheorem{ex}{Example}[section]

\newcommand{\E}{\mathbb E}

\newcommand{\bi}{\mathbf i}

\newcommand{\<}{\langle}
\renewcommand{\>}{\rangle}

\ifpdf
  \DeclareGraphicsExtensions{.eps,.pdf,.png,.jpg}
\else
  \DeclareGraphicsExtensions{.eps}
\fi

\newcommand{\TheTitle}{Wasserstein Hamiltonian Flow  with common noise on graph} 
\newcommand{\TheAuthors}{Jianbo, Cui and Shu, Liu and Haomin, Zhou}

\headers{}{\TheAuthors}

\title{{\TheTitle}\thanks{}}

\title{{\TheTitle}\thanks{The research is partially supported by Georgia Tech Mathematics Application Portal (GT-MAP) and by research grants NSF  DMS-1830225, and ONR N00014-21-1-2891. The research of the first author is
partially supported by start-up funds (P0039016) from Hong Kong Polytechnic University and the CAS AMSS-PolyU Joint Laboratory of Applied Mathematics.
}}

\author{Jianbo Cui 
\thanks{Department of Applied Mathematics, The Hong Kong Polytechnic University, Hung Hom, Kowloon, Hong Kong 
(\email{jianbo.cui@polyu.edu.hk}, corresponding author)
}
\and 
Shu Liu
\thanks{School of Mathematics, Georgia Institute of Technology, Atlanta, GA 30332, USA
(\email{sliu459@gatech.edu})}
\and 
Haomin Zhou
\thanks{School of Mathematics, Georgia Institute of Technology, Atlanta, GA 30332, USA
(\email{hmzhou@gatech.edu})}
}

\usepackage{amsopn}

\ifpdf
\hypersetup{
  pdftitle={\TheTitle},
  pdfauthor={\TheAuthors}
}
\fi

\usepackage{ulem}
\usepackage{subfigure}
\usepackage{graphicx}
\usepackage{enumerate}
\usepackage{amsmath,bm}
\usepackage{latexsym}
\usepackage{float}
\allowdisplaybreaks[4]

\begin{document}

\maketitle

\begin{abstract}
We study the Wasserstein Hamiltonian flow with a common noise on the density manifold of a finite graph. Under the framework of stochastic variational principle, we first develop the formulation of stochastic Wasserstein Hamiltonian flow and show the local existence of a unique solution. We also establish a sufficient condition for the global existence of the solution. Consequently, we obtain the global well-posedness for the nonlinear Schr\"odinger equations with common noise on graph. In addition, using Wong-Zakai approximation of common noise, we prove the existence of the minimizer for an optimal control problem with common noise. We show that its minimizer satisfies the stochastic Wasserstein Hamiltonian flow on graph as well. 
\end{abstract}

\begin{keywords}
stochastic Hamiltonian flow on graph, density manifold, Wong--Zakai approximation, optimal transport. 
\end{keywords}

\begin{AMS}
58B20,58J65,35Q41,49Q20. 
\end{AMS}

\section{Introduction}

The Wasserstein Hamiltonian flow defined on the cotangent bundle of probability density manifold, also known as Wasserstein manifold in the literature, has been studied in the past few years (see, e.g., \cite{MR924776,MR2361303,MR2808856,CP17}). Its relationship with the Hamiltonian ordinary differential equations (ODEs) has also been well demonstrated via optimal transport theory (see, e.g., \cite{Vil09,CLZ19,CLZ20}). Furthermore, it has been used in the theoretical or numerical analysis of nonlinear Schr\"odinger equation (see, e.g., \cite{Nelson19661079,MR783254,MR870196,CLZ19,CLZ20a}), mass optimal transport (see, e.g., \cite{BB00,GLM19,CLZ20a,CLZ21}) and the Schr\"odinger bridge problem (see, e.g., \cite{Leo14,leger2019hopfcole,CLMZ21,CLZ21a}). Extending Wasserstein Hamiltonian flow to account for random perturbations is challenging, because not all types of noise can be used to perturb the dynamics on density manifold in which the non-negativity of probability density function and mass conservation must be preserved. Recently, using the concept of common noise, also referred as the environment or system noise \cite{MR3572323,MR3753660}, a stochastic version of Wasserstein Hamiltonian flow is introduced to understand the collective dynamical behavior on density manifold of the stochastic Hamiltonian ODE defined on continuous phase space \cite{CLZ21s}.
However, little is known if the underlying space becomes discrete, such as a finite graph or a spatial discretization of a continuous space, due to several significant challenges that arise in the discrete space. 

Unlike the continuous space, whereas stochastic Hamiltonian ODEs can be identified and interpreted as the particle dynamics corresponding to the stochastic Wasserstein Hamiltonian flow, such a particle correspondence has not been established in the discrete space, which prevents adopting many well-developed techniques to the discrete case. For example, the particle version of stochastic Hamiltonian ODEs has been used as a push-forward map, a crucial tool in the analysis, to study the dynamical properties on the density manifold \cite{CLZ20}.  
This tool is hard to be generalized to a general graph partially because not all graph can be embedded into a continuous space \cite{CLZ21a}. Due to the loss of particle formulation, it is still unclear what kind of noise or random perturbation on finite graph can be used as a functional replacement of the white noise in the continuous space. In addition, low regularity of noise and the discrete structure of the graph make it harder to analyze the dynamical properties of Hamiltonian system on graph.   

In this paper, we propose two different strategies to establish the Wasserstein Hamiltonian flow with common noise on finite graph and investigate their mathematical properties. The first approach is based on the discrete version of generalized stochastic variational principle, which provides a formulation to construct stochastic Wasserstein Hamiltonian flow with given initial values. We use the stopping time technique to show its local well-posedness. Using Poisson bracket, we provide a sufficient condition on the energy terms in the variational principle to ensure the global well-posedness for the resulting system. We further demonstrate that both nonlinear Schr\"odinger equation and logarithmic Schr\"odinger equations with common noise on graph satisfy this sufficient condition. Thus they possess global solutions uniquely. In this consideration, it is observed that the Fisher information plays a fundamental role in obtaining the global existence result.

The second approach to derive the boundary value formulation of Wasserstein Hamiltonian flow with common noise on graph is proposed in the framework of stochastic optimal control. Using Wong--Zakai approximation \cite{MR195142,MR3712946} of common noise and von Neumann's minimax theorem \cite{MR1512442}, we prove the existence of minimizer for the stochastic optimal control problems. Under suitable assumptions, we show that their critical point satisfies Wasserstein Hamiltonian flow with common noise on the graph. In addition, the system obtained by the stochastic optimal control approach exhibits highly consistent formulation as those constructed by using stochastic variational principle. Yet, they have interesting differences, especially when the local well-posedness for the later one is no longer valid. In our investigation, these two strategies are complementary to each other in exploring the properties of stochastic Wasserstein Hamiltonian flow on graph.

The organization of this paper is as follows. In Section 2, we discuss what the common noise is and why it is used in our study. In Section 3, we review the basic notations    
of the deterministic Wasserstein Hamiltonian flow on a finite graph and the discrete optimal transport theory. In section 4, we present the discrete generalized stochastic variational principle to derive the stochastic Wasserstein Hamiltonian flow on graph and study several properties of the stochastic Wasserstein Hamiltonian flow.  In section 5, we give an alternative way based on stochastic optimal control to derive the stochastic Wasserstein Hamiltonian flow on graph. Meanwhile, we show the existence of the minimizer and derive its equation using Wong-Zakai approximation. 

\section{Common noise}
In this section, we borrow some examples to explain what common noise is and why it is a good choice for us to consider here. 

The first example is a mean-field game model (see, e.g. \cite{MR3572323}). Consider a $N$-player differential game, the state of each player $X_i(t)$ is a stochastic process described by a stochastic differential equation (SDE)
\begin{align*}
    dX_i(t)=b(t,X_i(t),\mu(t),\alpha_i)dt +\sigma(t,X_i(t),\mu(t))dB_i(t)+\sigma_0(t,X_i(t),\mu(t))dW(t),
\end{align*}
where $b, \sigma, \sigma_0$ are given functions, $\alpha_i$ is a control variable, $\mu(t)=\frac 1{N} \sum_{j=1}^N \delta_{X_i(t)}$, and $B_i (i=1 ,\cdots, N)$ and $W$ are one-dimensional independent Brownian motion defined on a completed probability space $(\Omega,\mathbb P,\mathcal F)$. In this model, the Brownian motion $B_i$ is called the idiosyncratic noise, which is introduced to model random perturbations to each individual, while $W$ is a stochastic perturbation independent of individuals and it is used to model the common disturbance to all players, hence it is called common noise. When $N\to\infty$, $\mu$ tends to a random measure reflecting the aggregate behavior of all players. $\mu$ is independent of $B_i$ while depending on the common noise $W$, because the effect from $B_i$ is averaged out but not for $W$. In this sense, $\mu(t)$ is a random measure flow perturbed by the common noise $W$. Conditioned on $W$, the model recovers the standard mean field game formulation (see the pioneering works \cite{MR2295621,MR2346927}). 

The second example is the stochastic nonlinear Schr\"odinger equation emerged from nonlinear optics, hydrodynamics, and plasma physics. For instance, in the molecular monolayers arranged in Scheibe aggregates \cite{PhysRevE.49.4627,PhysRevE.63.025601}, the thermal fluctuations of the phonons are included, which results in a stochastic nonlinear dynamical model given by
\begin{align*}
du=\bi \Delta u dt+\lambda \bi |u|^2u dt+\bi u\circ dW_t.   
\end{align*}
Here $\lambda\in \mathbb R$ is a constant, $W$ is a Wiener process on an infinite dimensional space,
and $\circ$ means that the stochastic integral is taken in the Stratonovich sense. The numerical experiments based on this stochastic model coincide with results reported in \cite{Chem} when temperatures are lower than $3K$.

Another model, called nonlinear  Schr\"odinger equation with random dispersion 
\begin{align*}
du=\bi \Delta u\circ dW_t+\bi \lambda |u|^2u,    
\end{align*}
is proposed to describe the propagation of signal (see e.g. \cite{Agra01b}), in which $W$ is a standard one-dimensional Brownian motion. In a recent study \cite{CLZ21s}, by using the Madelung transformation $u=\sqrt{\rho}e^{\bi S}$ and the stochastic variational principle on the density manifold, it is found that the mathematically equivalent systems in terms of $\rho$, $S$, and $W$ for above two stochastic nonlinear Schr\"odinger equations can be established. Under this viewpoint, $W$ is a random noise acting on the density function $\rho$. Therefore, it is a common noise, because it perturbs the entire density, not an individual particle.  In the mean-field game model and nonlinear Schr\"odinger equations, both $\mu$ and $\rho$ remain to be probability density functions, despite of the perturbations by the common noise $W$. In other words, non-negativity as well as mass can be preserved under common noise perturbations. Inspired by those examples, we select common noise to establish the stochastic Wasserstein Hamiltonian flow on graphs.

\section{Discrete optimal transport and discrete Wasserstein Hamiltonian flow}
In this section, we introduce the notations and some known results for the discrete optimal transport problem and Wasserstein Hamiltonian flow \cite{CLZ19,CLZ20a}. 

Consider a graph $G=(V,E,\omega)$ with a node set $V=\{a_i\}_{i=1}^N$, an edge set $E$, and $\omega_{jl}$ are the weights of the edges: $\omega_{jl}=\omega_{lj}>0$, if there is an edge between $a_j$ and $a_l$, and $0$ otherwise. Below, we will write $(i,j)\in E$ to denote the edge in $E$ between the vertices $a_i$ and $a_j$. Throughout the paper, we assume that $G$ is an undirected, connected graph with no self loops or multiple edges.

Let us denote the set of discrete probabilities on the graph by ${\mathcal{P}}(G)$:
$$\mathcal P(G)=\{(\rho)_{j=1}^N\ :\, \sum_{j}\rho_j =1, \rho_j\ge 0,\; \text{for} \; j\in V\},$$ 
and let $\mathcal P_o(G)$ be its interior (i.e., all  $\rho_j> 0$, for $a_j\in V$).
Let $\mathbb V_j$ be a linear potential on each node $a_j$, and $\mathbb W_{jl}=\mathbb W_{lj}$ an 
interactive potential between nodes $a_j,a_l$. The total linear potential $\mathcal V$ and interaction potential $\mathcal W$ are given by
$$\mathcal V(\rho)=\sum_{i=1}^N\mathbb V_i\rho_i,\,\,
\mathcal W(\rho)=\frac 12\sum_{i,j}\mathbb W_{ij}\rho_i\rho_j.$$ 

We let $N(i)=\{a_j\in V: (i,j)\in E\}$ be the adjacency set of node $a_i$ and 
$\theta_{ij}(\rho)$ be the density dependent weight on the edge $(i,j)\in E$. Consider the probability weight $\theta$ which is defined by $\theta_{ij}(\rho)=\Theta(\rho_i,\rho_j)$ with a continuous differentiable function $\Theta:[0,\infty)\times [0,\infty)\to [0,\infty)$ satisfying 
\begin{align*}
&\Theta\in \mathcal C^{\infty}((0,\infty)\times (0,\infty));\\
& \Theta \; \text{is continuous on}\; [0,\infty)\times [0,\infty);\\
&\Theta(s,t)=\Theta(t,s); \; \Theta(s,t)>0, \;\text{if}\; s,t>0,\\
& \min(s,t) \le \Theta(s,t)\le \max(s,t), s,t\ge 0.\\
& \Theta \; \text{is concave on} \; (0,\infty)\times (0,\infty).
\end{align*}
Two typical examples are the average function $\theta^A_{ij}(\rho_i,\rho_j)=\frac {\rho_i+\rho_j}2$ and the logarithmic mean $\theta^{L}_{ij}(\rho_i,\rho_j)=\frac {\log(\rho_i)-\log(\rho_j)}{\rho_i-\rho_j}.$ For more choices of the probability weight functions, we refer to \cite{CHLZ12,Mas11,CLMZ21,CLZ20a}.

Define the discrete Lagrange functional on the graph by 
\begin{equation}\label{DiscLag}
\mathcal L (\rho,v)=\int_0^1 \bigl[
\frac 12\<v,v\>_{\theta(\rho)}-\mathcal V(\rho)-\mathcal W(\rho)+\alpha L(\rho)-\beta I(\rho)\bigr] dt,
\end{equation}
where: $\rho(\cdot)\in \mathcal P_o(G)$,   the vector field $v$ is a skew-symmetric matrix on $E$. The inner product of two vector fields  $u,v$ is defined by 
$$\<u,v\>_{\theta(\rho)}:=\frac 12\sum_{(j,l)\in E}u_{jl}v_{jl}\theta_{jl}(\rho)\omega_{ij}.$$ 
The parameter $\beta\ge 0$,
the {\emph{discrete Fisher information}} \cite{CLZ19,CLZ21s} is defined by 
\begin{equation}\label{DiscFisher}
I(\rho)=\frac 12\sum_{i=1}^N\sum_{j\in N(i)}\widetilde \omega_{ij}|\log(\rho_i)-\log(\rho_j)|^2\widetilde \theta_{ij}(\rho),
\end{equation}
and the {\emph{discrete entropy}} is 
\begin{align*}
L(\rho)=\sum_{i=1}^N(\log(\rho_i)\rho_i-\rho_i),
\end{align*}
where $\alpha\in\mathbb R$, $\widetilde \omega,$ $\widetilde \theta$ can be another pair of weight  and
density dependent weight on $G$.
For convenience, we consider a typical case that  $\theta_{ij}(\rho)=\frac {\rho_i+\rho_j}2,\widetilde \theta_{ij}(\rho)=\frac {\rho_i-\rho_j}{\log(\rho_i)-\log(\rho_j)}$ in this paper.

The overall goal of discrete variational problem is to find the minimizer of $\mathcal L(\rho,v)$ subject to the discrete continuity equation on graph
\begin{align*}
\frac {d\rho_i}{d t}+div_G^{\theta}(\rho v)=0,
\end{align*}
where the discrete divergence of the flux function $\rho v$  is defined as 
$$div_G^{\theta}(\rho v):=-(\sum_{l\in N(j)}\sqrt{\omega_{jl}}v_{jl}\theta_{jl}).$$
As shown in \cite{CLZ20a}, the critical point $(\rho, v)$ of $\mathcal L$ satisfies 
$v=\nabla_G S:=\sqrt{\omega_{jl}}(S_j-S_l)_{(j,l)\in E}$ for some function $S$ defined on $G$. As a consequence,
the minimization problem leads to 
the following discrete  Wasserstein-Hamiltonian vector field on the graph $G$:
\begin{equation}\label{dhs}\begin{split}
&\frac {d\rho_i}{d t}+\sum_{j\in N(i)}\omega_{ij}(S_j-S_i)\theta_{ij}(\rho)=0,\\
&\frac {d S_i}{dt}+\frac 12\sum_{j\in N(i)}\omega_{ij}(S_i-S_j)^2 \frac {\partial \theta_{ij}(\rho)}{\partial \rho_i}+\beta \frac {\partial I(\rho)}{\partial \rho_i}-\alpha \log(\rho_i) +\mathbb V_i+\sum_{j=1}^N\mathbb W_{ij}\rho_j=0.
\end{split}\end{equation}
With respect to the variables $\rho$ and $S$, we can rewrite \eqref{dhs} as a Hamiltonian system with
Hamiltonian function 
$\mathcal H(\rho,S)=\mathcal K(S,\rho)+\mathcal F(\rho),$ where 
$ \mathcal K(S,\rho):=
\frac 12 \<\nabla_G S, \nabla_G S\>_{\theta(\rho)}$ and $\mathcal F(\rho):=\beta I(\rho)+\mathcal V(\rho)+\mathcal W(\rho).$ 
In particular, if $\beta=0,$ $\mathcal V=0,$ and $\mathcal W=0$, the infimum of 
$2\mathcal L(\rho,v)$ induces the Wasserstein metric on a finite graph, which is a discrete version of  Benamou-Brenier formula \cite{MR3834701}:
\begin{align*}
W(\rho^0,\rho^1):=\inf_{v}\Big\{\sqrt{\int_{0}^1 \<v,v\>_{\theta(\rho)}}dt \,\ : \,
\frac{d\rho}{dt}+div_G^{\theta}(\rho v)=0, \; \rho(0)=\rho^0,\; \rho(1)=\rho^1\Big\}.
\end{align*} 

\section{Discrete Wasserstein Hamiltonian flow with common noise}

In this section, we first use the discrete version of generalized stochastic variational principle in \cite{CLZ21s} to derive the discrete Wasserstein Hamiltonian flow with common noise. Then we study both the local and global existence of the unique solution for the stochastic Wasserstein Hamiltonian flow on graph.

Let us briefly introduce the generalized stochastic variational principle or Hamiltonian principle as follows. Define $W_{\delta}$ the linear Wong--Zakai approximation \cite{MR195142} of a standard Wiener process $W$, i.e. $W_{\delta}(t)=W_{\delta}(t_k)+\frac {t-t_k}{\delta}(W_{\delta}(t_{k+1})-W_{\delta}(t_k))$ for $t\in [t_k,t_{k+1})$ with $t_k=k\delta$, on a complete filtered probability space $(\Omega,\mathbb P,(\mathcal F)_{t\ge 0},\mathcal F)$.  Define the dominated energy and perturbed energy  as
\begin{align*}
&\mathcal H_0(\rho,S)=\mathcal K(S,\rho)+\mathcal F(\rho)-\alpha L(\rho),\;\\
&\mathcal H_1(\rho,S)=\eta_1 \mathcal K(S,\rho)+\eta_2 I(\rho)+\eta_3\mathcal V(\rho)+\eta_4\mathcal W(\rho)-\eta_5 L(\rho) 
\end{align*}
with different noise intensities $\eta_i\in \mathbb R, i=1,\cdots,5.$ We would like to remark that by taking different values for the noise intensities, the above general form covers many well-known problems, such as the stochastic optimal transport on graph, the stochastic Schr\"odinger equation, and Schr\"odinger equation with white noise on graph. Consider the following stochastic variational principle with Wong--Zakai approximation $W_{\delta},$
\begin{align}\label{gen-var-pri}
   \mathcal I(\rho^0,\rho^T)=\inf\{\mathcal S(\rho_t,\Phi_t)| (-\Delta_{\rho_t})^{\dagger}\Phi_t \in \mathcal T_{\rho_t} \mathcal P_{o}(G),\rho(0)=\rho^0, \rho(T)=\rho^T\}
\end{align}

whose action functional is given by the dual coordinates, 
\begin{align*}
\mathcal S(\rho_t,\Phi_t)&=\<\rho(0),\Phi(0)\>-\<\rho(T),\Phi(T)\>+\int_0^T \<\partial_t \Phi(t),\rho_t\>+\mathcal H_0(\rho_t, \Phi_t) dt  \\
&+\int_0^T \mathcal H_1(\rho_t,\Phi_t)\dot W_{\delta}dt.
\end{align*}
Here $\<\cdot,\cdot\>$ is the standard inner product in $\mathbb R^N$, and $(-\Delta_{\rho})^\dagger$ is the pseudo inverse of $div_G^{\theta}(\rho \nabla(\cdot))$, and $\mathcal T_{\rho} \mathcal P_{o}(G)$ is the tangent space at $\rho\in \mathcal P_{o}(G)$.  In particular, when $\eta_1=0,$ the above generalized 
Hamiltonian principle becomes the classical variational problem with random potential in Lagrangian formalism. 

By using Lagrange multiplier method,  one may verify that the critical point of \eqref{gen-var-pri} satisfies the following discrete stochastic Wasserstein Hamiltonian flow,
\begin{align*}
&\frac {d\rho}{d t}=\frac {\partial}{\partial S}\mathcal H_0(\rho,S)+\frac {\partial}{\partial S}\mathcal H_1(\rho,S)dW_{\delta}(t),\\
&\frac {d S}{dt}=-\frac {\partial}{\partial \rho}\mathcal H_0(\rho,S)-\frac {\partial}{\partial \rho}\mathcal H_1(\rho,S) dW_{\delta}(t).
\end{align*}
Moreover, if $\rho^0,\rho^T$ are $\mathcal F_0$ and $\mathcal F_T$ measurable functions and $\mathcal H_0,\mathcal H_1$ satisfies some growth conditions as those given in \cite{CLZ21s},  the limit of the above Wasserstein Hamiltonian flow with Wong--Zakai approximation converges to the stochastic Hamiltonian flow in Stratonovich sense.
\begin{align}\label{sdhs}
&\frac {d\rho}{d t}=\frac {\partial}{\partial S}\mathcal H_0(\rho,S)+\frac {\partial}{\partial S}\mathcal H_1(\rho,S)\circ dW(t),\\\nonumber
&\frac {d S}{dt}=-\frac {\partial}{\partial \rho}\mathcal H_0(\rho,S)-\frac {\partial}{\partial \rho}\mathcal H_1(\rho,S)\circ dW(t).
\end{align}
However, we would like to remark that it is difficult to rigorously show that \eqref{sdhs} is the critical point of \eqref{gen-var-pri} when $\delta \to 0.$

\subsection{Properties of Wasserstein Hamiltonian flow with common noise}

In this part, we consider the initial value problem of Eq. \eqref{sdhs} with $\rho(0)\in \mathcal P_{o}(G)$ which is $\mathcal F_{0}$--measurable. 
Let us first consider the local well-posedness of  Eq. \eqref{sdhs}.
For simplicity, we only present the detailed proof when 
$\theta_{ij}(\rho)=\frac {\rho_i+\rho_j}2,\widetilde \theta_{ij}(\rho)=\frac {\rho_i-\rho_j}{\log(\rho_i)-\log(\rho_j)}$ since the proof for the general case is analogous. 

\begin{prop}\label{local-well}
Let  $\rho(0)\in \mathcal P_{o}(G)$ and $S(0)\in \mathbb R^N$ be $\mathcal F_{0}$--measurable. Then there exists a stopping time $\tau^*(\rho(0),S(0))>0$ such that either
\begin{align*}
\tau^*(\rho(0),S(0))=+\infty, \; \text{or} \; \lim_{t\to \tau^*} \min_{i=1}^N\rho_i(t)=0\; \text{or} \; \lim_{t\to\tau^*} S(t)=\infty,  \;\text{a.s.}\;
\end{align*}
\end{prop}

\begin{proof}

Let $c>1.$
Denote smooth truncation functions $\theta^1,\theta^2$ such that 
\begin{align*}
&\theta^1_c(x):=1, \, x\in [0,c],\; \theta^1_c(x)=0, x\in [2c,\infty),\\
&\theta^2_c(x):=1, \; x\in [1/c,1],\; \theta^2_c(x)=0, x\in [0,1/{2c}].
\end{align*} 
The support of $\theta_c^1$ is chosen as $[0,2c]$ and that of $\theta_c^2$ is $[\frac 1{2c},1].$ 
Define $\phi_c^1(S,t),\phi_c^2(\rho,t)$ by 
\begin{align*}
\phi_c^1(S,t)=\theta^1_c(\|S\|_{\mathcal C([0,t]; \mathbb R^N)}), \;\;
\phi_c^2(\rho,t)=\theta^2_c\Big(\min_{i=1}^N \min_{s\in [0,t]}\rho_i(s)\Big),
\end{align*}
Notice that  
\begin{align*}
\frac {\partial \mathcal H_0}{\partial S_i}&=\sum_{j\in N(i)}(S_i-S_j)\frac {\rho_i+\rho_j}2,\\
\frac {\partial \mathcal H_0}{\partial \rho_i}&=\frac 14\sum_{j\in N(i)}(S_i-S_j)^2+\frac {\partial \mathcal F}{\partial \rho_i}-\alpha \log(\rho_i),\\
\frac {\partial \mathcal H_1}{\partial S_i}&= {\eta_1}\sum_{j\in N(i)}(S_i-S_j)\frac {\rho_i+\rho_j}2,\\
\frac {\partial \mathcal H_1}{\partial \rho_i}&=\frac {\eta_1}4\sum_{j\in N(i)}(S_i-S_j)^2+\eta_2\sum_{j\in N(i)}  \Big(\log(\rho_i)-\log(\rho_j)+\frac {\rho_i-\rho_j}{\rho_i}\Big)\\
&\quad+\eta_3 \mathbb V_i+\eta_4\sum_{j=1}^N \mathbb W_{ij}\rho_j-\alpha \frac 1{\rho_i}.
\end{align*}
Due to the relationship between It\^o integral and Stratonovich integral,
we consider the following truncated equation with $c>0$ large enough,
\begin{align}\label{trun-whf-ito}
\frac {d\rho^{c}}{d t}&=\phi_c^1(S^c,t) \phi_c^2(\rho^c,t) \frac {\partial}{\partial S}\mathcal H_0(\rho^c,S^c) dt+\phi_c^1(S^c,t) \phi_c^2(\rho^c,t) \frac {\partial}{\partial S}\mathcal H_1(\rho^c,S^c) dW_t\\\nonumber
&\quad-\frac 12\phi_c^1(S^c,t) \phi_c^2(\rho^c,t) \frac {\partial^2}{\partial S^2}\mathcal H_1(\rho^c,S^c) \frac {\partial }{\partial \rho} \mathcal H_1(\rho^c,S^c)dt\\\nonumber
&\quad+\frac 12\phi_c^1(S^c,t) \phi_c^2(\rho^c,t) \frac {\partial^2}{\partial \rho \partial S}\mathcal H_1(\rho^c,S^c) \frac {\partial }{\partial S} \mathcal H_1(\rho^c,S^c) dt \\\nonumber 
\frac {d S^{c}}{dt}&=-\phi_c^1(S^c,t) \phi_c^2(\rho^c,t) \frac {\partial}{\partial \rho}\mathcal H_0(\rho^c,S^c)-\phi_c^1(S^c,t) \phi_c^2(\rho,t) \frac {\partial}{\partial \rho}\mathcal H_1(\rho^c,S^c) dW_t\\\nonumber
&\quad+\frac 12\phi_c^1(S^c,t) \phi_c^2(\rho^c,t) \frac {\partial^2}{\partial S\partial \rho}\mathcal H_1(\rho^c,S^c)\frac {\partial }{\partial \rho} \mathcal H_1(\rho^c,S^c)dt \\\nonumber
&\quad-\frac 12\phi_c^1(S^c,t) \phi_c^2(\rho^c,t) \frac {\partial^2}{\partial \rho^2 }\mathcal H_1(\rho^c,S^c)\frac {\partial }{\partial S} \mathcal H_1(\rho^c,S^c) dt, 
\end{align}
where
\begin{align*}
(\frac {\partial^2}{\partial S^2}\mathcal H_1(\rho,S))_{ii}&=\eta_1 \sum_{j\in N(i)}\frac {\rho_i+\rho_j}2,\quad (\frac {\partial^2}{\partial S^2}\mathcal H_1(\rho,S))_{ij}=-\eta_1\frac {\rho_i+\rho_j}2,  \\
(\frac {\partial^2}{\partial S\partial \rho}\mathcal H_1(\rho,S))_{ii}&=\sum_{j\in N(i)}\frac{\eta_1}2(S_i-S_j), \quad
(\frac {\partial^2}{\partial S\partial \rho}\mathcal H_1(\rho,S))_{ij}=\frac{\eta_1}2(S_i-S_j),\\
(\frac {\partial^2}{\partial \rho^2 }\mathcal H_1(\rho,S))_{ii}&=\eta_4\mathbb W_{ii}+\eta_2(\sum_{j\in N(i)}\frac {1}{\rho_i}+\frac {\rho_j}{\rho_i^2})-\eta_5 \frac 1{\rho_i},
\\
(\frac {\partial^2}{\partial \rho^2 }\mathcal H_1(\rho,S))_{ij}&=\eta_4\mathbb W_{ij}-\eta_2(\frac 1{\rho_j}+\frac 1{\rho_i}).
\end{align*}
The local Lipschitz continuity of $\mathcal H_0(\rho^c,S^c)$ and $\mathcal H_1(\rho^c,S^c)$ implies the existence and uniqueness of the global mild solution for the truncated equation by the standard  arguments in \cite{oksendal2003stochastic}. Thus, for any $T>0$, there always exists a global mild solution $(\rho^c,S^c)\in \mathcal C([0,T];\mathbb R^N)\times \mathcal C([0,T];\mathbb R^N\times \mathbb R^N)$.
Now we define the local solution of equation \eqref{trun-whf-ito} as follows. For $n\in \mathbb N^+,$ define the stopping time $\tau_n$ by 
\begin{align*}
\tau_n:=\inf\{t\in [0,T]: \|S^n\|_{\mathcal C([0,t];\mathbb R^N\times \mathbb R^N)}\ge n\}\wedge \inf\{t\in [0,T]: \min_{i=1}^N \min_{s\in [0,t]}\rho_i^n(s)\le \frac 1{n}\},
\end{align*}
and $\tau_{\infty}:=\sup_{n\in \mathbb N}\tau_n.$ This is guaranteed by the fact that $$Z^n(t):=\|S^n\|_{\mathcal C([0,t];\mathbb R^N\times \mathbb R^N)}+\frac 1{\min_{i=1}^N \min_{s\in [0,t]}\rho_i^n(s)}<\infty$$ 
defines an increasing, 
continuous and $\mathcal F_t$-adapted process with $Z^n(0)=\frac 1{\min_{i=1}^N\rho_i(0)}.$ 

For $n\le k,$ set $\tau_{k,n}:=\inf\{t\in [0,T]: Z^k(t)\ge n\}$. Then we have $\tau_{k,n}\le \tau_k$ and thus $\phi^1_n(S^k,t)=\phi^1_k(S^k,t)=1$, $\phi_n^2(\rho^k,t)=\phi_k^2(\rho^k,t)$ on $\{t\le \tau_{k,n}\}.$ This leads to $(\rho^k,S^k)=(\rho^n,S^n)$ and $Z^k=Z^n$ a.s. on $\{t\le \tau_{k,n}\}.$ 
We conclude that $\tau_{k,n}=\tau_n,$ a.s. and define the local solution $(\rho,S)$ up to the stopping time $\tau_{\infty}$ by 
 $(\rho,S)=(\rho^n,S^n),$ on $\{t\le \tau_n\}.$ 
\end{proof}

We would like to mention that in \cite{CLZ19,CLZ20a}, the global solution in deterministic case ($\eta_1=\cdots=\eta_5=0$) is obtained by using the energy conservation law if $\mathcal F(\rho)$ contains the Fisher information $\beta I(\rho), \beta>0$. In stochastic case, the existence of global solution becomes more complicate and depends on the relationship between the deterministic energy $\mathcal H_0$ and $\mathcal H_1$.

To see this fact, applying It\^o's formula on $\mathcal H_0$, we obtain that before $\tau^n,$
it holds that
\begin{align}\label{Ito}
\mathcal H_0(\rho^n(t),S^n(t))
&=\mathcal H_0(\rho^n(0),S^n(0))
+\int_0^t\frac {\partial \mathcal H_0}{\partial \rho}^{\top}\frac {\partial\mathcal H_1}{\partial S}dW_s-\int_0^t\frac {\partial\mathcal H_0}{\partial S}^{\top}\frac {\partial\mathcal H_1}{\partial \rho}dW_s\\\nonumber
&-\frac 12 \int_0^t (\frac {\partial \mathcal H_0}{\partial \rho}^{\top} \frac {\partial^2 \mathcal H_1}{\partial S^2} \frac {\partial \mathcal  H_1}{\partial \rho}-
\frac {\partial \mathcal H_0}{\partial \rho}^{\top} \frac {\partial^2 \mathcal H_1}{\partial S\partial \rho} \frac {\partial \mathcal  H_1}{\partial S})ds\\\nonumber
&
+\frac 12 \int_0^t (\frac {\partial \mathcal H_0}{\partial S}^{\top} \frac {\partial^2 \mathcal H_1}{\partial S \partial \rho} \frac {\partial \mathcal  H_1}{\partial \rho}-
\frac {\partial \mathcal H_0}{\partial S}^{\top} \frac {\partial^2 \mathcal H_1}{\partial \rho^2} \frac {\partial \mathcal  H_1}{\partial S})ds\\\nonumber
&+\int_0^t \frac 12 (\frac {\partial^2 \mathcal H_0}{\partial \rho^2})\cdot(\frac {\partial \mathcal H_1}{\partial S},\frac {\partial \mathcal H_1}{\partial S} )ds\\\nonumber
&-\int_0^t  (\frac {\partial^2 \mathcal H_0}{\partial \rho\partial S})\cdot(\frac {\partial \mathcal H_1}{\partial S},\frac {\partial \mathcal H_1}{\partial \rho} )ds\\\nonumber
&+\int_0^t \frac 12(\frac {\partial^2 \mathcal H_0}{\partial S^2})\cdot(\frac {\partial \mathcal H_1}{\partial \rho},\frac {\partial \mathcal H_1}{\partial \rho} ) ds.
\end{align}
To get the global existence of the solution, it suffices to show 
 $$\sup_{n}\mathbb E\Big[\sup_{s\in [0,\tau^n]}\mathcal H_0(\rho(s),S(s)) \Big]< \infty.$$ 
Therefore, a sufficient condition to ensure the global existence of the solution is that
\begin{align*}
|\{\mathcal H_0,\mathcal H_1\}
|+|\{\mathcal H_1,\{\mathcal H_0,\mathcal H_1\}\}|&\le c_1\mathcal H_0+C_1\;\; \text{for some} \;c_1,C_1>0,
\end{align*}
where $\{\cdot,\cdot\}$ is the Poisson bracket.
In particular, when $\{\mathcal H_0,\mathcal H_1\}=0,$ $\mathcal H_0$ is an invariant of the stochastic Wasserstein Hamiltonian flow. A typical example is $\mathcal H_0$ being a multiple of $\mathcal H_1.$

\begin{tm}\label{well-shs}
Let  $\beta>0$, $\alpha\in \mathbb R$, $T>0,$ $\rho(0)\in \mathcal P_{o}(G)$, and $S(0)\in \mathbb R^d$ be $\mathcal F_{0}$--measurable and have the finite second moment. 
Assume that there exists $c_1,C_1>0$ such that 
$$|\{\mathcal H_0,\mathcal H_1\}
|+|\{\mathcal H_1,\{\mathcal H_0,\mathcal H_1\}\}|\le c_1\mathcal H_0+C_1$$
Then there exists a unique global solution of \eqref{sdhs} satisfying $\rho(t)\in \mathcal P_{o}(G), t\in [0,T].$  
\end{tm}

\begin{proof}
It suffices to prove that $\rho(t)\in \mathcal P_{o}(G), t\in [0,T].$ Let $\alpha \ge 0$. By applying \eqref{Ito}, taking expectation, and employing  the Burkerholder inequality, we achieve that 
\begin{align*}
\E[\sup_{t\in [0,T]}\mathcal H_0(\rho(t),S(t))]
&\le \E[\mathcal H_0(\rho(0),S(0))]+C \int_0^t\E[\mathcal H_0(\rho(s),S(s))+C_1']ds\\
&+C\E[\Big(\int_0^t(\mathcal H_0(\rho(s),S(s))+C_1')^2ds\Big)^{\frac 12}].
\end{align*}
The Gronwall's inequality leads that 
\begin{align*}
\E[\sup_{t\in [0,T]}\mathcal H_0(\rho(t),S(t))]&\le C(T,\rho(0),S(0)).
\end{align*}
It follows that $\sup\limits_{t\in [0,T]}\mathcal H_0(\rho(t),S(t))<\infty, a.s.$ Due to the fact that 
$$\mathcal H_0(\rho,S)=\mathcal K(\rho,S)+\beta I(\rho)+V(\rho)+W(\rho)-\alpha L(\rho),$$ there always exists a constant $C>0$ such that $\mathcal H_0(\rho,S)+C>0.$ 
Therefore, if $\beta>0$, $\alpha\ge 0$, we obtain that 
\begin{align*}
\sup\limits_{t\in [0,T]} \mathcal K(\rho(t),S(t))<\infty, \text{a.s.}, \;\text{and}\; \sup\limits_{t\in [0,T]} I(\rho(t))<\infty, \text{a.s.}
\end{align*}
Let us define  $C_{Kin}:=\sup\limits_{t\in [0,T]} \mathcal K(\rho(t),S(t)).$
The fact that $I$ is positive infinity on the boundary and the definition of $\mathcal K$ yield that 
\begin{align*}
&\min_{t\in [0,T]}\min_{i=1}^N\rho_i(t)>0, a.s.\\
& \max_{ij\in E}|S_i(t)-S_j(t)|\le  \sqrt{ \frac {C_{Kin}}{\min\limits_{t\in [0,T]}\min\limits_{i=1}^N\rho_i(t)} }<\infty, a.s.
\end{align*}
We conclude that  $\rho_i(t)\in \mathcal P_o(G), t\in [0,T],$ a.s.
The case that $\alpha<0$ can be proven by the similar steps and the fact that $x\log(x)$ is uniformly bounded in $[0,1].$

\end{proof}

\begin{rk}
Theorem \ref{well-shs} still holds for more general Hamiltonian system which does not contain the Fisher information but any other potential $\mathbb Z(\rho)$ provided that $\mathbb Z(\rho)$ is smooth, bounded from below in $\mathcal P_o(G)$, and it approaches infinity at the boundary of $\mathcal P(G)$ . We also would like to remark that other choices of $\theta,\widetilde \theta$ are available (see, e.g., \cite{CLZ20a}).
\end{rk}

As shown in Theorem 4.1, the lower bound of the density $\rho$ is a positive random variable, a.s.. We end this section with three examples of stochastic nonlinear Schr\"odinger equation on $G.$ When $G$ is a lattice graph, the following examples can be viewed as the spatial approximations of the stochastic nonlinear Schr\"odinger equation on continuous space \cite{CHL16b}.

\begin{ex}
[Nonlinear Schr\"odinger equation on graph with common noise \cite{PhysRevE.49.4627,PhysRevE.63.025601,MR1425880,CHL16b}] 
When nonlinear Schr\"odinger equation on graph is perturbed by the common noise, it reads 
\begin{align}\label{snls1}
\bi \frac {d u_j}{dt}=-\frac 12 (\Delta_G u)_j+u_j\mathbb V_j+u_j\sum_{l=1}^N\mathbb W_{jl}|u_l|^2+\sigma_ju_j\circ dW_t.
\end{align}
Here $\sigma$ is a potential on $G$, $\Delta_G$ the nonlinear Laplacian operator on $G$ \cite{CLZ19,CLZ20a} is  defined by 
\begin{align*}
(\Delta_G u)_j&=-u_j\Big(\frac 1{|u_j|^2}\Big[\sum_{l\in N(j)} \omega_{jl}\theta_{jl}(\Im(\log(u_j))-\Im(\log(u_l))) \\
&+\sum_{l\in N(j)}\widetilde \omega_{jl} \widetilde \theta_{jl}(\Re(\log(u_j))-\Re(\log(u_l)) \Big] \\
&+\sum_{l\in N(j)} \omega_{jl} \frac {\partial \theta_{jl}}{\partial \rho_j}|\Im(\log(u_j)-\log(u_l))|^2 \\
&+\sum_{l\in N(j)}\widetilde \omega_{jl} \frac {\partial \widetilde \theta_{jl}}{\partial \rho_j}|\Re(\log(u_j)-\log(u_l))|^2
\Big),
\end{align*}
where $\Im$ and $\Re$ are the imaginary and real parts of a complex number, respectively.   

Denoting  the complex form $u_j=\sqrt{\rho_j}e^{\bi S_j(t)}, j=1,\cdots,N,$ the Madelung system on graph becomes,
\begin{align}\label{snls-mad}
& d\rho_i=\sum_{j\in N(i)}\omega_{ij}(S_i-S_j)\theta_{ij}(\rho) dt;\\\nonumber
& d S_i+(\sum_{j\in N(i)}\frac 12 \omega_{ij}(S_i-S_j)^2 \frac {\partial \theta_{ij}}{\partial \rho_i}+\frac 18 \frac {\partial}{\partial \rho_i} I(\rho)+\mathbb V_i
+\sum_{j\in N(i)}\mathbb W_{ij}\rho_j)dt+\sigma_idW_t=0.
\end{align} 
It can be verified that the nonlinear Schr\"odinger equation on graph with common noise satisfies the condition of Theorem \ref{well-shs}. Consequently, there exists a global unique solution for \eqref{snls1}.

\end{ex}

\begin{ex}[Logarithmic Schr\"odinger equation with common noise on graph \cite{CS21}] 
The logarithmic Schr\"odinger equation on graph  perturbed by the common noise is 
\begin{align}\label{snls10}
\bi \frac {d u_j}{dt}=-\frac 12 (\Delta_G u)_j+u_j\mathbb V_j+u_j\sum_{l=1}^N\mathbb W_{jl}|u_l|^2-u_j\log(|u_j|^2)+\sigma_ju_j\circ dW_t.
\end{align}
Here $\sigma$ is a potential on $G$.
Denoting  the complex form $u_j=\sqrt{\rho_j}e^{\bi S_j(t)}, j=1,\cdots,N,$ the Madelung system on graph follows,
\begin{small}
\begin{align*}
d\rho_i=\sum_{j\in N(i)}\omega_{ij}(S_i-S_j)\theta_{ij}(\rho) dt;&\\
d S_i+(\sum_{j\in N(i)}\frac{\omega_{ij}}{2}(S_i-S_j)^2 \frac {\partial \theta_{ij}}{\partial \rho_i}+\frac 18 \frac {\partial}{\partial \rho_i} I(\rho)+\mathbb V_i &
+\sum_{j\in N(i)}\mathbb W_{ij}\rho_j-\log(\rho_i))dt+\sigma_idW_t=0.
\end{align*}
\end{small}
One can verify that this nonlinear Schr\"odinger equation on graph with common noise satisfies the condition of Theorem \ref{well-shs} with $\mathcal H_0(\rho,S)=\mathcal K(\rho,S)-L(\rho)+\mathcal V(\rho)+\mathcal W(\rho)+\frac 18 I(\rho)$ and $\mathcal H_1(\rho,S)=\sum_{i=1}^N\rho_i\sigma_i.$
As a consequence, there exists a global unique solution for \eqref{snls10}.

\end{ex}

\begin{ex}
[White noise dispersion nonlinear Schr\"odinger equation on graph \cite{Agra01b,Agra01a}] 
The stochastic dispersive Schr\"odinger equation reads
 \begin{align}\label{snls2}
& d\rho_i=\sum_{j\in N(i)}\omega_{ij}(S_i-S_j)\theta_{ij}(\rho) \circ dW_t;\\\nonumber
& d S_i+(\frac 12\sum_{j\in N(i)}\omega_{ij}(S_i-S_j)^2 \frac {\partial \theta_{ij}}{\partial \rho_i}+\frac 18 \frac {\partial}{\partial \rho_i} I(\rho))\circ dW_t+(\mathbb V_i
+\sum_{j\in N(i)}\mathbb W_{ij}\rho_j)dt=0,
\end{align} 
which is equivalent to 
\begin{align*}
\bi \frac {d u_j}{dt}=-\frac 12 (\Delta_G u)_j\circ dW_t+(u_j\mathbb V_j+u_j\sum_{l=1}^N\mathbb W_{jl}|u_l|^2)dt.
\end{align*}
One may repeat the proof of Theorem \ref{well-shs} and show that the  modification energy $\widetilde {\mathcal H_0}(\rho,S)=\mathcal K(\rho,S)+\frac 18 I(\rho)$ is bounded, a.s..  Thus there exists  uniquely  a global solution for \eqref{snls2}.
\end{ex}

We note that these examples are constructed by the critical points of stochastic variational principles and have not been considered before. To the best of our knowledge, the existing results on the existence of solutions for these three examples are obtained in continuous spaces.

We end this section by summarizing that for any initial values, the Wasserstein Hamiltonian flow with common noise established here has local well-posedness up to a positive stopping time. However, it is not clear whether the system is the critical point of a stochastic variational principle. Furthermore, it is technically challenging to directly analyze the existence and uniqueness of the minimizer, even if a stochastic variational principle can be identified. To address the shortcomings of this approach, we propose another approach based on optimal control formulation to construct the boundary value formulation of Wasserstein Hamiltonian flow with common noise in the next section.

\section{Optimal control problem with common noise}

Here we consider the following variational principle in the framework of optimal control
\begin{align}\label{opt-com}
&\inf_{\rho,v}\; [\int_0^1 \frac 12\<v_t,v_t\>_{\theta(\rho_t)} dt]\\\nonumber 
& \text{subject to:}\;\; d\rho(t)+div_G^{\theta}(\rho(t) v(t))+div_G^{\theta}(\rho(t)\nabla_G \Sigma)\circ dW_t=0\\\nonumber
& \text{and}\;\; \rho(0,\omega)=\rho_a,\; \rho(1,\omega)=\rho_b,
\end{align}
where $\Sigma$ is a given vector field on $G,$ $\rho_a$ and $\rho_b$ are given $\mathcal F_0$-measurable and $\mathcal F_{1}$-measurable densities in $\mathcal P_o(G).$ Via a discrete Hopf--Cole transform (see e.g. \cite{CLZ21a}), one can show that the critical point of \eqref{opt-com} formally coincides with that of the discretization of stochastic Schr\"odinger bridge problem in \cite{CLZ21s}, i.e.,
\begin{align*}
\inf\{S(\rho_t,\Phi_t): (-\Delta_{\rho_t})^\dagger \Phi_t\in \mathcal T_{\rho_t}\mathcal P_o(G), \rho(0)=\rho^a, \rho(1)=\rho^b\}.
\end{align*}
Recall that $ (-\Delta_{\rho_t})^\dagger$ is the pseudo inverse of $div_G^{\theta}(\rho\nabla_G(\cdot))$, $\mathcal T_{\rho_t}\mathcal P_o(G)$ is the tangent space of $\mathcal P_o(G)$ at $\rho_t$
and
\begin{align*}
\mathcal S(\rho_t,\Phi_t)&=\<\rho(0),\Phi(0)\>-\<\rho(1),\Phi(1)\>+\int_0^1 \<\partial_t \Phi(t),\rho_t\>+\mathcal H_0(\rho_t, \Phi_t) dt \\
&\quad + \int_0^1\mathcal H_1(\rho_t,\Phi_t)\circ dW(t).
\end{align*}
with $\mathcal H_0(\rho,S)=\frac 14\sum_{ij\in E} (S_i-S_j)^2\theta_{ij}(\rho), \mathcal H_1(\rho,S)=\frac 12\sum_{ij\in E}(\Sigma_i-\Sigma_j)(S_i-S_j)\theta_{ij}(\rho)$.
By the Lagrangian multiplier method, the critical point of \eqref{opt-com}, if exists, is expected to satisfy
\begin{small}
\begin{align*}
 &d\rho_i(t)+\sum_{j\in N(i)} \theta(\rho_i,\rho_j)(S_j-S_i)dt+\sum_{j\in N(i)} \theta(\rho_i,\rho_j)(\Sigma_j-\Sigma_i) \circ dW_t=0,\\
&dS_i(t)+\sum_{j\in N(i)} \frac 12(S_j-S_i)^2 \frac {\partial \theta}{\partial \rho_i}(\rho_i,\rho_j)dt+\sum_{j\in N(i)}  (S_i-S_j)(\Sigma_i-\Sigma_j)\frac {\partial \theta}{\partial \rho_i}(\rho_i,\rho_j) \circ dW_t=0.
\end{align*}
\end{small}
However, due to the low regularity of $W$, it seems difficult to directly show the existence of the minimizer of \eqref{opt-com}. To overcome the challenges, we consider an optimal control problem perturbed by Wong-Zakai approximations of the Wiener process.

\subsection{Optimal control perturbed by Wong--Zakai approximations}
In this part, we prove the existence of the minimizer of the optimal control problem with Wong--Zakai approximations, which is formulated as 
\begin{align}\label{ot-graph-wk}
&\inf_{\rho,v}\; [\int_0^1 \frac 12\<v_t,v_t\>_{\theta(\rho_t)} dt]\\\nonumber 
& \text{subject to:}\;\; d\rho(t)+div_G^{\theta}(\rho(t) v(t))+div_G^{\theta}(\rho(t)\nabla_G \Sigma) dW_t^{\delta}=0\\\nonumber
& \text{and}\;\; \rho(0,\omega)=\rho_a,\; \rho(1,\omega)=\rho_b.
\end{align}
    
It should be mentioned that the critical points of \eqref{ot-graph-wk} and \eqref{gen-var-pri}, if exist, share the same equation. However, it is not clear how to obtain the existence of the minimizer of \eqref{gen-var-pri}, which motivates us to investigate the minimizer of \eqref{ot-graph-wk}. To this end,
we first illustrate that the value of \eqref{ot-graph-wk} is finite.

Given $\rho^a,\rho^b\in \mathcal P(G)$, 
we define the feasible set $C_{F}(\rho^a,\rho^b)$ of pairs $(\rho,m)$ 

\begin{align*}
  C_F(\rho^a, \rho^b) & = \Bigg\{ \rho\in H^{1}([0,1];\mathcal P(G)), m\in L^2([0,1];\mathcal S^{N\times N}) \Bigg| 
     (\rho(0),\rho(1))=(\rho^a,\rho^b), \\
    &\qquad d\rho_i(t)+\underset{j\in N(i)}{\sum} m_{ij}dt + \underset{j\in N(i)}{\sum}(\Sigma_j-\Sigma_i)\theta_{ij}(\rho) dW^{\delta}(t)=0.
\Bigg\}
\end{align*}
Here $\mathcal S^{N\times N}$ denotes the skew-symmetric matrix, $N$ is the node number. 
We consider an equivalent form of \eqref{ot-graph-wk}, i.e., 
$\inf\limits_{\rho,m} \mathcal A(\rho,m)$ over the set $C_{F}(\rho^a,\rho^b),$
where 
\begin{align}\label{equ-ot-graph}
\mathcal A(\rho,m):=\int_0^1 \frac 14\sum_{(i,j)\in E} L(\theta_{ij}(\rho),m_{ij}) dt,
\end{align}
$L(x,y)=\frac {y^2}x$ if $x>0$, $L(x,y)=0$ if $x=y=0$ and $L(x,y)=\infty$ otherwise.
The equivalence between \eqref{ot-graph-wk} and $\inf_{\rho,m}$\eqref{equ-ot-graph} is based on the following reasons.

$\eqref{ot-graph-wk}\ge \inf_{\rho,m} \eqref{equ-ot-graph}$: this part is straightly forward by defining $m_{ij}=\theta_{ij}(S_i-S_j).$ When $\theta_{ij}=0,$ define $m_{ij}=0.$

$\eqref{ot-graph-wk}\le \inf_{\rho,m} \eqref{equ-ot-graph}$: For any fixed $\rho$, denote $v_{ij}=\frac {m_{ij}}{\theta_{ij}},$ and $\mathbb H_{\rho}=\{[v] \subset \mathcal S^{N\times N}|w \in [v] \; \text{if an only if}\; v_{ij}=w_{ij}\;  \text{for} \; \theta_{ij} (\rho)\neq 0\}. $ Under the graph inner product $\<\cdot, \cdot\>_{\theta(\rho)},$ $\mathbb H_{\rho}$ forms a finite dimensional subspace. Thus $\nabla_G$ defines a linear map from the potential functional space (consider $L^2(G)$ such that it is also an Hilbert space) to $\mathbb H_{\rho}$, and $div_G$ defines a map from the matrix space to $L^2(G).$ Denote $P_{\rho}$ the orthogonal projection in $\mathbb H_{\rho}$ onto the range of $\nabla_G.$
Then for any feasible path $(\rho_t,m_t), t\in [0,1]$ in \eqref{equ-ot-graph}, one can always find a potential functional $S_t$ such that $P_{\rho_t} v_t=\nabla_G S_t$. Thanks to the fact that $\mathbb H_{\rho}=Ran(\nabla_G)\otimes Ker(div_G)$, we have that $(I-P_{\rho_t})v_t\in Ker(div_G)$ and thus $div(\rho_t v_t)=div_G(\rho_t\nabla_G S_t).$ As a consequence, $(\rho_t,S_t)$ also belongs to the feasible set of \eqref{ot-graph-wk}.

\begin{prop}\label{non-empty-wk}
For any $\rho^a,\rho^b \in \mathcal P(G),$ there is a path $(\rho,m)\in C_{F}(\rho^a,\rho^b)$ such that $\mathcal A(\rho,m)<\infty.$
\end{prop}
\begin{proof}
We use an induction argument on the number of nodes in $G$.  First, consider the case that the cardinality of $V=\{1,2\}$ is 2, the edge $E=\{(1,2),(2,1)\}$ and $\rho^a\neq \rho^b.$
Define $\rho_1(t)=\rho_1^a, t\in [0,1-\delta], \rho_1(t)=\rho_1^a+(\rho_1^b-\rho_1^a) \frac {t-1+\delta}{\delta}, t\in [1-\delta, 1].$
Then it follows that 
\begin{align*}
\rho_1(t)-\rho_1(0)=\int_0^t m_{21}(s) ds+\int_0^t \frac 12(\Sigma_1-\Sigma_2)dW^{\delta}(s).
\end{align*}
Therefore, we get  
\begin{align*}
m_{21}(t)&=\frac 12(\Sigma_2-\Sigma_1) \dot W^{\delta}(t),\; t\in [0,1-\delta],\\
m_{21}(t)&=(\rho_1^b-\rho_1^a) \frac {1}{\delta}+\frac 12(\Sigma_2-\Sigma_1) \dot W^{\delta}(t),\; t\in [1-\delta, 1],
\end{align*}
where $\dot W^{\delta}(t)=\frac {W(t_{k+1})-W(t_k)}{\delta}, t_k=k\delta, k\le K-1, K\delta=1, t\in[t_k,t_{k+1}].$
Notice that 
\begin{align*}
&\int_0^1m_{21}^2(s)ds\\
&=\frac 14\int_{0}^{1-\delta}(\Sigma_2-\Sigma_1)^2 (\dot W^{\delta}(t))^2ds+\int_{1-\delta}^1[(\rho_1^b-\rho_1^a) \frac {1}{\delta}+\frac 12(\Sigma_2-\Sigma_1) \dot W^{\delta}(t)]^2ds\\
&\le \frac 14(\Sigma_2-\Sigma_1)^2 \sum_{k=0}^{K-1} \frac {(W_{t_{k+1}}-W_{t_k})^2}{\delta}+\frac 14(\Sigma_2-\Sigma_1)^2\frac {(W_{t_K}-W_{t_{K-1}})^2}{\delta}\\
&+\frac {(\rho^b_1-\rho^a_1)^2}{\delta}+(\rho_1^b-\rho_1^a) (\Sigma_2-\Sigma_1) \frac {W_{t_K}-W_{t_{K-1}}}{\delta}\le C(\delta)<\infty,\; \text{a.s.}
\end{align*}
This covers the case $n=2$.
When $n>2,$ 
we use the concatenation arguments to show the finiteness of 
\eqref{equ-ot-graph}. Namely, we need show that if there exists $\rho\in \mathcal P(G)$ such that $C_F(\rho^a,\rho)$ and $C_F(\rho,\rho^b)$ have feasible paths then $C_F(\rho^a,\rho^b)$ also has an feasible path. Because for any $\rho_a,\rho_b$, we can set an intermediate state $(0,\cdots,0,1)$ and show that there are feasible paths connecting $\rho_a$ and $(0,\cdots,0,1)$,  $(0,\cdots,0,1)$ and $\rho_b$, respectively. By integrating these two paths continuously, we could construct a feasible path from $\rho_a$ to $\rho_b$.

Without loss of generality, we may assume that $\rho^b=(0,\cdots,0,1)$. If the support of $\rho^a$ is the same as $\rho^b$, then it follows that $\rho^a=\rho^b, (\rho,m)\in C_F(\rho^a,\rho^b)$ as long as 
\begin{align*}
\sum_{j\in N(i)}m_{ij}(t)=\sum_{j\in N(i)}(\Sigma_i-\Sigma_j)\theta(\rho_i,\rho_j) \dot W^{\delta}(t).
\end{align*} 
Supposing that the support of $\rho^a$ has an intersection with the first $N-1$ nodes, we iteratively construct a sequence $\widetilde \rho^0,\cdots,\widetilde \rho^{l_0}$ satisfying that $
\widetilde \rho^0=\rho^a, \; \widetilde \rho^{l_0}=\rho^b,$ the cardinality of the support of $\widetilde \rho^{l}$ is strictly smaller than that of $\widetilde \rho^{l-1}$, and that there is a feasible path connecting $\widetilde \rho^{l-1}$ with $\widetilde \rho^l$ in the interval $[t_{l-1},t_l],$ where $t_l= \frac {l}{l_0}.$
\end{proof}

Introducing the corresponding saddle scheme formally, 
\begin{align*}
\inf_{\rho}\sup_{\lambda} \Big[\mathcal A(\rho,m)-\int_0^1\<\lambda, \dot \rho(t)+div_G^{\theta}(\rho(t) v(t))+div_G^{\theta}(\rho(t)\nabla_G \Sigma) \dot W_t^{\delta}\>dt\Big]
\end{align*}
with $\rho(0)=\rho^a$ and $\rho(1)=\rho^b$, 
it can be seen that there exists $\lambda\in BV_{loc}([0,1];\mathbb R^N)$ such that  the critical point $(\rho,v)$ of the \eqref{ot-graph-wk} satisfies
\begin{align*}
&\theta_{ij}(\rho)[v_{ij}-(\lambda_i-\lambda_j)]=0,\;\forall (i,j)\in E,\\
&\<\dot \lambda,\rho\>-\frac 14\sum_{ij}v_{ij}^2\theta_{ij}(\rho)+ \frac 12\sum_{ij}(\Sigma_i-\Sigma_j)(\lambda_{i}-\lambda_{j})\theta_{ij}(\rho)dW^{\delta}(t)=0, \; \mathcal L^1 \; \text{a.e.}
\end{align*}
Denote $S_i=-\lambda_i$. When the optimal path does not intersect the boundary of $\mathcal P(G),$
the above equations become the stochastic Wasserstein Hamiltonian flow (see e.g. \cite{CLZ21s}), 
\begin{align}\label{swhf-wk}
\dot \rho=\nabla_{S} \mathcal H(\rho,S)+\nabla_{S} \mathcal H_1(\rho,S)\dot W^{\delta},\;\\\nonumber
\dot S=-\nabla_{\rho} \mathcal H(\rho,S)-\nabla_{\rho} \mathcal H_1(\rho,S)\dot W^{\delta},
\end{align}
where $\mathcal H(\rho,S)=\frac 14\sum_{ij\in E} (S_i-S_j)^2\theta_{ij}(\rho), \mathcal H_1(\rho,S)=\frac 12\sum_{ij\in E}(\Sigma_i-\Sigma_j)(S_i-S_j)\theta_{ij}(\rho)$.
Indeed, we have the following result.

\begin{prop}
Let $\rho^a,\rho^b\in \mathcal P(G)$. Assume that $(\rho,m)\in C_F(\rho^a,\rho^b)$ and that $S\in H^1([0,1];\mathbb R^N)$ satisfies 
\begin{align*}
\<\dot S,\rho\>+\frac 14\sum_{ij}(S_i-S_j)^2\theta_{ij}(\rho)+ \sum_{ij}(\Sigma_i-\Sigma_j)(S_{i}-S_{j})\theta_{ij}(\rho)dW^{\delta}(t)\le 0, \; \mathcal L^1 \;\text{a.e.}
\end{align*} 
Then
\begin{enumerate}
\item [(i)] it holds that 
$$\<S(1),\rho^b\>-\<S(0),\rho^a\>\le \mathcal A(\rho,m).$$
\item [(ii)] Equality holds in $(i)$ if and only if 
\begin{align*}
&m_{ij}=\theta_{ij}(\rho)(\nabla_G S)_{ij}, \; \forall (i,j)\in E, \;  \\
&\<\dot S,\rho\>+\frac 14\sum_{ij}(S_i-S_j)\theta_{ij}(\rho)+ \sum_{ij}(\Sigma_i-\Sigma_j)(S_{i}-S_{j})\theta_{ij}(\rho)dW^{\delta}(t)= 0.
\end{align*}
\item [(iii)] If $\rho\in \mathcal P_o(G),$ a.e., then $(\rho,S)$ satisfies \eqref{swhf-wk}, a.e.
\end{enumerate}
\end{prop}

\begin{proof}
By using the integration by parts and H\"older's inequality, we get
\begin{align*}
&\<S(1),\rho^b\>-\<S(0),\rho^a\>\\
= & \int_0^1(\<m,\nabla_G S\>+\<\rho,\dot S\>)dt+\int_0^1\<\nabla_G \Sigma, \nabla_G S\>_{\theta(\rho)}dW^{\delta}(t)\\
\le & \mathcal A(\rho,m)+\int_0^1 \Big(\<\rho,\dot S\>+\frac 12 \|\nabla_G S\|_{\theta(\rho)}^2+\<\nabla_G \Sigma, \nabla_G S\>_{\theta(\rho)}\dot {W^{\delta}(t)}\Big)dt \le \mathcal A(\rho,m).
\end{align*} 
From the above estimate, the equality holds if and only if the conditions in (ii) hold.
If  $\rho\in \mathcal P_o(G),$ a.e., we obtain
\begin{align*}
0&=\<\rho,\dot S+\nabla_{\rho}\mathcal H(\rho,S)+\nabla_{\rho}\mathcal H_1(\rho,S)\dot W^{\delta}\>\\
&=\<\rho,\dot S\>+\mathcal H(\rho,S)+\mathcal H_1(\rho,S)\dot W^{\delta}\le 0,
\end{align*}
which completes the proof.
\end{proof}

Now we focus on the existence of the minimizer of \eqref{ot-graph-wk}.

\begin{tm}\label{tm-ver1}
Let $\rho^a,\rho^b\in \mathcal P(G).$ There exists $(\rho^*,v^*,m^*)$ such that 
$(\rho^*,v^*)$ minimizes \eqref{ot-graph-wk} and $(\rho^*,m^*)$ minimizes $\inf_{\rho,m} \mathcal A(\rho,m).$
\end{tm}

\begin{proof}
By Proposition \ref{non-empty-wk}, there exists a path $(\rho,m)\in C_F(\rho^a,\rho^b)$ such that $\mathcal A(\rho,m)\le C<\infty$ for some constant $C>0$, which implies that 
$\|m\|_{L^2([0,T];\mathcal S^{n\times n})}\le 2C.$ Then the equation of $\dot \rho$, together with the Poincar\'e--Wirtinger inequality, implies that 
$\rho\in H^1([0,1];\mathbb R^N).$
The intersection of $ C_F(\rho^a,\rho^b)$ with any sub-level set , i.e., $\{(\rho,S)| \mathcal A(\rho, \theta(\rho)\nabla_G S)\le c\}$ for some $c\ge 0,$ 
of $\mathcal A$ is a precompact set in the weak topology of $ H^1([0,1];\mathbb R^N)\times L^2([0,1];\mathcal S^{N\times N})$. Notice that $\mathcal A$ is non-negative and weakly lower semi-continuous on $ H^1([0,1];\mathbb R^N)\times L^2([0,1];\mathcal S^{N\times N})$ (see, e.g., \cite{GLM19}). Thus it achieves its minimum at some path $(\rho^*,m^*)\in C_F(\rho^a,\rho^b).$

Next we define a measurable vector field $v^*$ as $v^*_{ij}(t)=\frac {m^*_{ij}(t)}{\theta_{ij}(\rho)}$ if $\theta_{ij}(\rho)>0$, and $v^*_{ij}(t)=0$ otherwise. As a consequence, we have that $\frac 12\int_{0}^1\|v^*\|_{\theta(\rho)}^2dt=\mathcal A(\rho^*,m^*)<\infty$. Then we show that $(\rho^*,v^*)$ is also a minimizer of \eqref{ot-graph-wk}.
Let $(\rho,v)$ be a feasible set of \eqref{ot-graph-wk} and set $m_{ij}=\theta_{ij}(\rho)v_{ij}.$ It holds that $\mathcal A(\rho,m)=\frac 12 \int_0^1\|v\|^2_{\theta(\rho)}dt<\infty$ and $(\rho,m)\in C_F(\rho^a,\rho^b).$ From the property of $(\rho^*,m^*),$ we have $\int_0^1\|v^*\|^2_{\theta(\rho^*)}dt\le \int_0^1\|v\|^2_{\theta(\rho)}dt.$
\end{proof}

Now we are in a position to  show the following duality property
\begin{align}\label{dual}
\min_{(\rho,m)\in C_F(\rho^a,\rho^b)} \mathcal A(\rho,m)&=
\sup_{S}\Big\{\<S(1),\rho^b\>-\<S(0),\rho^a\>: 
 \sup_{\rho}\{\<\dot S,\rho\>+\frac 14\sum_{ij}v_{ij}^2\theta_{ij}(\rho)\\\nonumber
 &\quad+ \frac 12\sum_{ij}(\Sigma_i-\Sigma_j)(S_{i}-S_{j})\theta_{ij}(\rho)dW^{\delta}(t)\}=0\Big\}.
\end{align} 
The key is using the minimax identity of the following Lagrange multiplier,
\begin{align*}
\mathcal L(\rho,m,S)&:=\<S(1),\rho^b\>-\<S(0),\rho^a\>+\mathcal A(\rho,m)
\\
&\quad-\int_0^1(\<\dot S,\rho\>+\<m,\nabla_G S\>+\<\nabla_G \Sigma,\nabla_G S\>\dot W^{\delta}(t))dt.
\end{align*} 
To ensure the boundedness of $S$, we consider a subset $H^1_{R}$ of $H^1([0,1];\mathbb R^n)$ which is defined by $H^1_{R}:=\{S\in H^1([0,1];\mathbb R^n): \|S\|_{ H^1([0,1];\mathbb R^n)}\le R\}, R>0.$ We claim that the following property holds,
\begin{align}\label{minmax}
\inf_{(\rho,m)}\sup_{S\in H^1_R}\mathcal L(\rho,m,S)=\sup_{S \in H^1_R}\inf_{(\rho,m)}\mathcal L(\rho,m,S)
\end{align}
by applying the standard minimax theorem in \cite[Theorem I.1.1.]{MR3560551}. It suffices to prove that $H^1_R$ is convex and compact in the weak topology, $\mathcal A$ is convex in the weak topology, $\{S\in H_R^1: \mathcal L(\rho,m,S)\ge C\}$ is closed convex set in $H_R^1$, and $\{(\rho,m)\in \mathcal C_F(\rho^a,\rho^b): \mathcal L(\rho,m,S)\le C \}$ is a convex set for any $C\in\mathbb R$. All these conditions can be verified since $\mathcal L$ is convex in $(\rho,m)$ and linear in $\lambda,$ and that $H^1([0,1];\mathbb R^N)$ is compact in $L^2([0,1];\mathbb R^N).$
Furthermore, we also have that 
\begin{align}\label{form}
\sup_{S\in H_R^1} \mathcal L(\rho,m,S)
=\mathcal A(\rho,m)+R\mathcal E(\rho,m),
\end{align}
where the nonnegative functional $\mathcal E$ is defined by 
\begin{align*}
\mathcal E(\rho,m)&:=\sup_{S\in H_1^1}\{\<S(1),\rho^b\>-\<S(0),\rho^a\>
\\
&-\int_0^1(\<\dot S,\rho\>+\<m,\nabla_G S\>+\<\nabla_G \Sigma,\nabla_G S\>_{\theta(\rho)}\dot W^{\delta}(t))dt\}.
\end{align*}
It can been seen that $\mathcal E=0$ only if $(\rho,m)\in C_F(\rho^a,\rho^b)$ and larger than 0 otherwise.
By making use of the lower continuity and convexity of $\mathcal L$ and $\mathcal E$, it  can be seen that for any $R>0,$ there exists $(\rho^{*,R},m^{*,R})$ such that it minimizes $\mathcal A+R\mathcal E.$ Furthermore, the set $\{(\rho^{*,R},m^{*,R})\}_{R>0}$ is precompact, which complete the proof by taking $R \to \infty$.

\begin{lm}\label{lm-minmax1}
The commutative property holds,
\begin{align}\label{minmax1}
 \inf_{(\rho,m)}\sup_{S\in H^1}\mathcal L(\rho,m,S)=\sup_{S \in H^1}\inf_{(\rho,m)}\mathcal L(\rho,m,S).
 \end{align}
\end{lm}
\begin{proof}
Since $(\rho^*,m^*)\in  C_F(\rho^0,\rho^1)$, we have that for any $R>0,$
\begin{align*}
\mathcal A(\rho^*,m^*)=\sup_{S\in H_R^1}\mathcal L(\rho^*,m^*,S)
\ge \inf_{(\rho,m)}\sup_{S\in H_R^1} \mathcal L(\rho,m,S).
\end{align*}
By \eqref{minmax},
\begin{align*}
\mathcal A(\rho^*,m^*)\ge \sup_{S\in H_R^1}\inf_{(\rho,m)} \mathcal L(\rho,m,S).
\end{align*}
Recall that $(\rho^*,m^*)$ is the minimizer of the optimal control problem with common noise.
\eqref{form} and  $(\rho^*,m^*)\in  C_F(\rho^0,\rho^1)$ implies that 
\begin{align*}
    \mathcal A(\rho^*,m^*)&=\sup_{S\in \mathbb H^1_R}\mathcal L(\rho^*,m^*,S)\ge \inf_{\rho,m}\sup_{S\in \mathbb H^1_R}\mathcal L(\rho,m,S)
    \\
    &=\mathcal A(\rho^{*,R},m^{*,R})+R\mathcal E(\rho^{*,R},m^{*,R})\\
    &\ge \mathcal A(\rho^{*,R},m^{*,R}).
\end{align*}

Denote the accumulation point of $\{\rho^{*,R},m^{*,R}\}$ by $(\rho^{*,\infty},m^{*,\infty}).$ It follows that $(\rho^{*,\infty},m^{*,\infty})\in C_F(\rho^a,\rho^b)$ and therefore that
\begin{align*}
\mathcal A(\rho^*,m^*)\le \mathcal A(\rho^{*,\infty},m^{*,\infty}).
\end{align*}
We conclude that 
\begin{align*}
&\mathcal A(\rho^*,m^*)=\mathcal A(\rho^{*,\infty},m^{*,\infty}), \quad
\limsup_{R\to +\infty} R \mathcal E(\rho^{*,R},m^{*,R})=0.
\end{align*}
It suffices to prove $$\inf_{(\rho,m)\in C_F(\rho^0,\rho^1)}\sup_{S\in H^1}\mathcal L(\rho,m,S)\le \sup_{S \in H^1}\inf_{(\rho,m)\in  C_F(\rho^0,\rho^1)}\mathcal L(\rho,m,S).$$
By using \eqref{minmax}, we obtain that 
\begin{align*}
\mathcal A(\rho^{*,\infty},S^{*,\infty})\le \lim_{R\to \infty} \sup_{S\in H^1_R}\inf_{(\rho,m)} \mathcal L(\rho,m,S)\le \sup_{S\in H^1}\inf_{(\rho,m)} \mathcal L(\rho,m,S),
\end{align*}
and that 
\begin{align*}
\mathcal A(\rho^{*,\infty},S^{*,\infty})=\sup_{S\in H^1}\mathcal L(\rho^{*,\infty},m^{*,\infty},S)\ge\inf_{(\rho,m)}\sup_{S\in H^1} \mathcal L(\rho,m,S),
\end{align*}
which completes the proof.
\end{proof}

\begin{tm}\label{tm-dual}
The dual property \eqref{dual} holds.
\end{tm}

\begin{proof}
For any $(\rho,m)\in H^1([0,1],\mathbb R^N)\times L^2([0,1],\mathcal S^{N\times N})$, by \eqref{form}, we have 
\begin{align*}
\sup_{S\in H^1}\mathcal L(\rho,m,S)=\mathcal A(\rho,m)+\mathbb I_{C_F(\rho^a,\rho^b)}(\rho,m),
\end{align*}
where $\mathbb I_{C_F(\rho^a,\rho^b)}(\rho,m)=0$ if $(\rho,m)\in C_F(\rho^a,\rho^b)$, otherwise $\mathbb I_{C_F(\rho^a,\rho^b)}(\rho,m)=\infty.$
By using \eqref{minmax1}, we achieve that 
\begin{align*}
  \inf_{(\rho,m)}\sup_{S\in H^1}\mathcal L(\rho,m,S)
  &=
 \inf_{(\rho,m)} \{\mathcal A(\rho,m)+\mathbb I_{C_F(\rho^a,\rho^b)}(\rho,m)\}\\
 &= \inf_{(\rho,m)\in C_F(\rho^0,\rho^1)} \{\mathcal A(\rho,m)\}.
\end{align*}
Notice that for a fixed $S\in H^1,$ using the H\"older inequality, we get 
\begin{align*}
\inf_{(\rho,m)}\mathcal L(\rho,m,S)=\<S(1),\rho^b\>-\<S(0),\rho^a\>-\int_0^1\max(H(\dot S,\nabla_G S),0)dt,
\end{align*}
where $H(\dot S,\nabla_G S):=\sup_{\rho} \Big\{\<\dot S,\rho\>+\frac 12\|\nabla_G S\|^2_{\theta(\rho)}+\<\nabla_G \Sigma,\nabla_G S\>_{\theta(\rho)}\dot W^{\delta}(t))\Big\}.$
Thus it follows that 
\begin{align*}
\inf_{(\rho,m)\in C_F(\rho^0,\rho^1)} \{\mathcal A(\rho,m)\}=\sup_{S\in H^1}\inf_{(\rho,m)}\mathcal L(\rho,m,S)
\end{align*}
if $H(\dot S,\nabla_G S)\le 0 $, $\mathcal L^1$ a.e. 
It only needs to show the existence of $\bar S$ such that $H(\dot {\bar S},\nabla_G {\bar S})=0$ and that $\<\bar S(1),\rho^b\>-\<\bar S(0),\rho^a\>\ge \<S(1),\rho^b\>-\< S(0),\rho^a\>.$ To this end, let $\mathcal O:=\{H(\dot S,\nabla_G S)<0\}$ and assume that $\mathcal L^1(\mathcal O)>0.$ 
Define $\bar S_i=S_i+\alpha$ with $\alpha(t)=-\int_0^t\chi_{\mathcal O} H(\dot S,\nabla_G S)ds$. Thus, we get
\begin{align*}
& \; \<\bar S(1),\rho^b\>-\<\bar S(0),\rho^a\>-\int_{\mathcal O}H(\dot {\bar S},\nabla_G {\bar S})dt \\
&=\<S(1),\rho^b\>-\<S(0),\rho^a\>-\int_{\mathcal O} H(\dot S,\nabla_G S)dt\\
&\ge \<S(1),\rho^b\>-\<S(0),\rho^a\>,
\end{align*}
which completes the proof.
\end{proof}

Now, we are able to describe the Hamiltonian structure of the minimizer. 
Following the idea of \cite{GLM19}, define $\dot S=\dot S^{\textrm{sing}}+\dot S^{\textrm{abs}}$ where $\dot S^{\textrm{abs}}$ is the absolutely continuous part (w.r.t. $\mathcal{L}_1$) of $\dot S$ and $\dot S^{\textrm{sing}}$ is the singular part (w.r.t. $\mathcal{L}_1$) of $\dot S,$
then we have that 
\begin{align*}
&d\rho(t)+div_G^{\theta}(\rho(t) \nabla_G S(t))+div_G^{\theta}(\rho(t)\nabla_G \Sigma) dW^{\delta}_t=0,\\
&\<\dot S^{\textrm{abs}},\rho\>+\frac 14\sum_{ij}(S_i-S_j)^2\theta_{ij}(\rho)+ \sum_{ij}(\Sigma_i-\Sigma_j)(S_{i}-S_{j})\theta_{ij}(\rho)dW^{\delta}_t=0,\; \mathcal L^1 \; \text{a.e.} \\
&\<\frac {d \dot S^{\textrm{sing}}}{d\mu},\rho\>=0, \; \forall\; \mu\; \text{a.e.}, \; \mu\; \bot \;\mathcal L_1. 
\end{align*}
the singular means the singular part of $\dot S$ w.r.t. $\mathcal L_1$.
When the optimal path does not intersect the boundary of $\mathcal P(G),$ we recover Eq. \eqref{swhf-wk}.
We would like to remark that if the minimizer $(\rho,S)$ is also  predictable  (see e.g. \cite{DZ14}), then  Eq. \eqref{swhf-wk} converges to a stochastic Wasserstein Hamiltonian flow driven by the standard Brownian motion when $\delta\to0$ \cite{CLZ21s}.

In the deterministic case, the $\theta$-connected components has been introduced in \cite{Mas11,GLM19} to study whether the optimal transfer achieves the boundary of the density manifold in optimal transport on graph (see, e.g., \cite[Section 1]{Mas11}, \cite[Section 3]{GLM19}). In this part, we demonstrate that this approach may fail in the stochastic case, such as \eqref{ot-graph-wk}. Let $\rho\in \mathcal P(G).$ The nodes $i,j\in V$ are called $\theta$-connected, if there exists integers $i_1,\cdots,i_k\in V$ such that $i_1=i,i_k=j, (i_l,i_{l+1})\in E, l\le k-1$ and $\theta_{i_1i_2}(\rho)\cdots \theta_{i_{k-1}i_{k}}(\rho) > 0.$
The largest $\theta$-connected set containing $i$ is called the $\theta$-connected component of $i$. All the $\theta$-components of $\rho$
form a partition of $V.$

We use the following example to illustrate that $\theta$-connected component may not characterize the optimal path.

\begin{rk}
Let $V=\{1,2,3\}, E=\{(1,2),(2,3)\}.$ Let $\rho^a=(0,0,1)$ and $\rho^b=(0,\frac 12,\frac 12).$  We can not obtain that $\rho$ connecting $\rho^a$ and $\rho^b$ lies on the boundary as in the deterministic case.  
To see this fact, assume that $(\rho,m)\in C(\rho^a,\rho^b)$ with $\rho_1\not\equiv 0$. We have that 
\begin{align*}
\dot \rho_1+m_{12}&=(\Sigma_1-\Sigma_2)\theta_{12}(\rho)\dot W^{\delta},\\
\dot \rho_2+m_{21}+m_{23}&=(\Sigma_2-\Sigma_1)\theta_{21}(\rho)\dot W^{\delta}+(\Sigma_2-\Sigma_3)\theta_{23}(\rho)\dot W^{\delta},\\
\dot \rho_3+m_{32}&=(\Sigma_3-\Sigma_2)\theta_{32}(\rho)\dot W^{\delta}.
\end{align*}
Then one may define $(\widetilde \rho_1,\widetilde \rho_2,\widetilde \rho_3)=(0,\rho_1+\rho_2,\rho_3)$ and 
\begin{align*}
&\widetilde m_{12}=(\Sigma_1-\Sigma_2)\theta_{12}(\rho)\dot W^{\delta},\; \widetilde m_{23}-(\Sigma_2-\Sigma_3)\dot W^{\delta}=m_{23}-(\Sigma_2-\Sigma_3)\theta_{23}(\rho)\dot W^{\delta}.
\end{align*}
Then it holds that $\widetilde \rho(0)=\rho^a, \widetilde \rho(1)=\rho^b$ and $\dot {\widetilde \rho_1}=0.$ By the definition of $\widetilde \rho,$ it could be shown that 
\begin{align*}
&\dot {\widetilde \rho_2}+\widetilde m_{23}=(\Sigma_2-\Sigma_3)\dot W^{\delta},\\
&\dot {\widetilde \rho_3}+\widetilde m_{32}=(\Sigma_3-\Sigma_2)\dot W^{\delta}.
\end{align*}
Therefore, we have 
\begin{align*}
\mathcal A(\rho,m)=\frac 12\int_0^1\Big(\frac {m_{12}^2}{\theta_{12}(\rho)}+\frac {m_{23}^2}{\theta_{23}(\rho)}\Big)dt,
\end{align*}
and 
\begin{align*}
\mathcal A(\widetilde \rho,\widetilde m)=\frac 12\int_0^1\frac {((\Sigma_1-\Sigma_2)\theta_{12}(\widetilde \rho)\dot W^{\delta})^2}{\theta_{12}(\widetilde \rho)}+(m_{23}+\frac 12\rho_1(\Sigma_2-\Sigma_3)\dot W^{\delta})^2 dt.
\end{align*}
However, we may not have $\mathcal A(\widetilde \rho,\widetilde m)\ge \mathcal A(\rho,m).$
\end{rk}

In the next subsection, we construct an optimal control problem with a special stochastic perturbation such that $\theta$-connect method still holds in the stochastic case.

\subsection{Optimal control problem with a special stochastic perturbation}
Now we consider a special perturbation of optimal control problem, that is,
\begin{align}\label{ot-graph-sp}
&\inf_{\rho,v}\; [\int_0^1 \frac 12\<v_t,v_t\>_{\theta(\rho_t)} dt]\\\nonumber 
& \text{subject to:}\;\; d\rho(t)+div_G^{\theta}(\rho(t) v(t))+div_G^{\theta}(\rho(t)v(t)) dW_t^{\delta}=0\\\nonumber
& \text{and}\;\; \rho(0)=\rho_a,\; \rho(1)=\rho_b.
\end{align}
Note that \eqref{ot-graph-sp} is different from \eqref{ot-graph-wk} since the diffusion term in the constraint involves $v(t)$. In the continuous space, the constraint in the critical equation has a corresponding stochastic differential equation driven by the multiplicative noise in the particle level and thus is different from \eqref{ot-graph-wk} whose corresponding stochastic differential equation is driven by the additive noise. Another motivation is whether $\theta$-component  method can characterize 
the property that the optimal transfer touches the boundary as in the deterministic case.

Given $\rho^a,\rho^b\in \mathcal P(G)$, we define the feasible set $C_F(\rho^a,\rho^b)$ of pairs $(\rho,m)$ such that 
\begin{align*}
\rho\in H^{1}([0,1];\mathcal P(G)), m\in L^2([0,1];\mathcal S^{n\times n}), \; (\rho(0),\rho(1))=(\rho^a,\rho^b)
\end{align*}
and 
\begin{align*}
d\rho_i(t)+\sum_{j\in N(i)}m_{ij}dt+\sum_{j\in N(i)}m_{ij}dW^{\delta}(t) =0.
\end{align*}
We consider the equivalent form of \eqref{ot-graph-sp} 
$\inf_{\rho,m} \mathcal A(\rho,m)$ over the set $C_F(\rho^a,\rho^b)$, where $\mathcal A$ is defined in \eqref{equ-ot-graph}.

\begin{prop}\label{non-empty-wk1}
For any $\rho^a,\rho^b \in \mathcal P(G),$ there is a path $(\rho,m)\in S_F(\rho^a,\rho^b)$ such that $\mathcal A(\rho,m)<\infty.$
\end{prop}
\begin{proof}

The proof is similar to that of Proposition \ref{non-empty-wk}.
We use an introduction argument on the nodes number of $G$.  First, consider the case that the cardinality of $V=\{1,2\}$ is 2, the edge $E=\{(1,2),(2,1)\}$ and $\rho^a\neq \rho^b.$
Define $\rho_1(t)=\rho_1^a, t\in [0,1-\delta], \rho_1(t)=\rho_1^a+(\rho_1^b-\rho_1^a) \frac {t-1+\delta}{\delta}, t\in [1-\delta, 1].$
Then it follows that 
\begin{align*}
\rho_1(t)-\rho_1(0)=\int_0^t m_{21}(s)(1+\dot{W^{\delta}}(s)) ds.
\end{align*}
Therefore, we get  
\begin{align*}
m_{21}(t)&=0,\; t\in [0,1-\delta],\\
m_{21}(t)(1+\dot W^{\delta}(t))&=(\rho_1^b-\rho_1^a) \frac {1}{\delta},\; t\in [1-\delta, 1], \;\mathcal L^1 \; \text{a.e.}, 
\end{align*}
where $\dot W^{\delta}(t)=\frac {W(t_{k+1})-W(t_k)}{\delta}, t_k=k\delta, k\le K-1, K\delta=1, t\in[t_k,t_{k+1}].$
Notice that 
\begin{align*}
&\int_0^1m_{21}^2(s)ds\\
&=\int_{1-\delta}^1\frac 1{\delta^2}\frac {(\rho_1^b-\rho_1^a)^2}{(1+\dot{W^{\delta}}(s))^2}ds\\
&\le \int_{1-\delta}^1\frac {(\rho_1^b-\rho_1^a)^2}{(\delta+{W_{t_{K}}-W_{t_{K-1}}})^2 } ds
\\
&\le C(\delta)<\infty,\; \text{a.s.}
\end{align*}
This covers the case $n=2$.
When $n>2,$ 
we use the concatenation arguments to show the finiteness of 
$\mathcal A$ as in the proof of Proposition \ref{ot-graph-sp}. 
\end{proof} 

Applying the Lagrange multiplier method,   
it can be seen that  the critical point $(\rho,v)$ of \eqref{ot-graph-sp} will satisfy
\begin{align*}
&\theta_{ij}(\rho)[v_{ij}-(\lambda_i-\lambda_j)(1+\dot W^{\delta})]=0,\;\forall (i,j)\in E,\\
&\<\dot \lambda,\rho\>+\frac 14 \sum_{ij} v_{ij}^2\theta_{ij}(\rho)+\frac 12\sum_{ij}v_{ij}(\lambda_j-\lambda_i)\theta_{ij}(\rho)\\
&\quad +\frac 12\sum_{ij}v_{ij}(\lambda_j-\lambda_i)\theta_{ij}(\rho)dW^{\delta}(t)=0, \; \mathcal L^1 \; \text{a.e.}
\end{align*}
Denote $S_i=-\lambda_i$. When the optimal path does not intersect the boundary of $\mathcal P(G),$
the above equations become the stochastic Wasserstein Hamiltonian flow (see \cite{CLZ21s}), 
\begin{align}\label{swhf-wk1}
\dot \rho=\nabla_{S} \mathcal H(\rho,S)(1+\dot W^{\delta})^2,\;\\\nonumber
\dot S=-\nabla_{\rho} \mathcal H(\rho,S)(1+\dot W^{\delta})^2,
\end{align}
where $\mathcal H(\rho,S)=\frac 14\sum_{ij\in E} (S_i-S_j)^2\theta_{ij}(\rho)$.
Now we present the existence of the minimizer of \eqref{ot-graph-sp}  whose proof is similar to that of Theorem \ref{tm-ver1} and thus is omitted. 

\begin{tm}
Let $\rho^a,\rho^b\in \mathcal P(G).$ There exists $(\rho^*,v^*,m^*)$ such that 
$(\rho^*,v^*)$ minimizes \eqref{ot-graph-sp} and $(\rho^*,m^*)$ minimizes $\inf_{\rho,m} \mathcal A(\rho,m).$
\end{tm}

By similar steps in proving Theorem \ref{tm-dual}, we could obtain the following duality property
\begin{align}\label{dual1}
\min_{(\rho,m)\in C_F(\rho^a,\rho^b)} \mathcal A(\rho,m)&=
\sup_{S}\Big\{\<S(1),\rho^b\>-\<S(0),\rho^a\>: 
 \sup_{\rho}\{\<\dot S,\rho\> \\\nonumber
 &+ \frac 14\sum_{ij}(S_{i}-S_{j})^2\theta_{ij}(\rho)(1+\dot{W^{\delta}}(t))^2\}=0\Big\}.
\end{align} 
The key step is  proving the minimax identity 
\begin{align}\label{minmax11}
 \inf_{(\rho,m)}\sup_{S\in H^1}\mathcal L(\rho,m,S)=\sup_{S \in H^1}\inf_{(\rho,m)}\mathcal L(\rho,m,S)
 \end{align}
of
\begin{align*}
\mathcal L(\rho,m,S)&:=\<S(1),\rho^b\>-\<S(0),\rho^a\>+\mathcal A(\rho,m)
\\
&\quad-\int_0^1(\<\dot S,\rho\>+\<m,\nabla_G S\>(1+\dot W^{\delta}(t)))dt,
\end{align*} 
which is analogous to that of Lemma \ref{lm-minmax1}.
As a consequence, the arguments in the proof of Theorem \ref{tm-dual} lead to the following theorem.

\begin{tm}\label{tm-dual1}
The dual property \eqref{dual1} holds.
\end{tm}
 
We omitted the tedious details here, since it follows the same approach as the one given the proof of Theorem \ref{tm-dual}.

Define $\dot S=\dot S^{\textrm{sing}}+\dot S^{\textrm{\textrm{abs}}}$ where $\dot S^{\textrm{abs}}$ is the absolutely continuous part of $\dot S$ and $\dot S^{\textrm{sing}}$ is the singular part of $\dot S$, then we have that 
\begin{align*}
&d\rho(t)+div_G^{\theta}(\rho(t) \nabla_G S(t))(dt+dW_t^{\delta})=0,\\
&\<\dot S^{\textrm{abs}},\rho\>+\frac 14\sum_{ij}(S_i-S_j)^2\theta_{ij}(\rho)(1+\dot{ W_t^{\delta}})^2=0,\; \mathcal L^1 \; \text{a.e.} \\
&\<\frac {d\dot S^{\textrm{\textrm{sing}}}}{d\mu},\rho\>=0, \; \forall \mu\; \text{a.e.}, \; \mu\; \bot \;\mathcal L_1. 
\end{align*}
When the optimal path does not intersect the boundary of $\mathcal P(G),$ we recover Eq. \eqref{swhf-wk1}.

We would like to point out that  in this particular case, we can use the $\theta$-connected components to study whether the optimal transfer achieves the boundary of the density manifold in optimal transport on graph.  We use the following example to illustrate the reason.

\begin{rk}
Let $V=\{1,2,3\}, E=\{(1,2),(2,3)\}.$ Let $\rho^a=(0,0,1)$ and $\rho^b=(0,\frac 12,\frac 12).$  We claim that $\rho$ connecting $\rho^a$ and $\rho^b$ lies on the boundary as in the deterministic case.  
To see this fact, assume that $(\rho,m)\in C_F(\rho^a,\rho^b)$ with $\rho_1\not\equiv 0$. We have that 
\begin{align*}
\dot \rho_1+m_{12}(1+\dot W^{\delta})&=0,\\
\dot \rho_2+(m_{21}+m_{23} )(1+\dot W^{\delta})&=0,\\
\dot \rho_3+m_{32}(1+\dot W^{\delta})&=0.
\end{align*}
Then one may define $(\widetilde \rho_1,\widetilde \rho_2,\widetilde \rho_3)=(0,\rho_1+\rho_2,\rho_3)$ and 
$
\widetilde m_{12}=0,\; \widetilde m_{23}=m_{23}.
$
Then it holds that $\widetilde \rho(0)=\rho^a, \widetilde \rho(1)=\rho^b$ and $\dot {\widetilde \rho_1}=0.$ By the definition of $\widetilde \rho,$ it could be shown that 
\begin{align*}
&\dot {\widetilde \rho_2}+\widetilde m_{23}(1+\dot W^{\delta})=0,\\
&\dot {\widetilde \rho_3}+\widetilde m_{32}(1+\dot W^{\delta})=0.
\end{align*}
Therefore, we have 
\begin{align*}
\mathcal A(\rho,m)=\frac 12\int_0^1\Big(\frac {m_{12}^2}{\theta_{12}(\rho)}+\frac {m_{23}^2}{\theta_{23}(\rho)}\Big)dt,
\end{align*}
and 
\begin{align*}
\mathcal A(\widetilde \rho,\widetilde m)=\frac 12\int_0^1\frac {m_{23}^2}{\theta_{23}(\widetilde \rho)} dt=\frac 12\int_0^1 m_{23}^2 dt.
\end{align*}
We have $\mathcal A(\widetilde \rho,\widetilde m)<\mathcal A(\rho,m),$ which leads to a contradiction.
\end{rk}

Next we show the relationship between \eqref{ot-graph-sp} with a small perturbation  $\epsilon \dot W^{\delta}$ and the classical optimal transport problem. 
By defining $\hat v=v(1+\epsilon \dot W^{\delta}),$ then \eqref{ot-graph-sp} can be rewritten as 
\begin{align}\label{ot-graph-sp-lim}
&\inf_{\rho,\widehat v}\; [\int_0^1 \frac 12 \frac 1{(1+\epsilon \dot W^{\delta})^2}\<\widehat v_t,\widehat v_t\>_{\theta(\rho_t)} dt]\\\nonumber 
& \text{subject to:}\;\; d\rho(t)+div_G^{\theta}(\rho(t) \widehat v(t))=0\\\nonumber
& \text{and}\;\; \rho(0)=\rho_a,\; \rho(1)=\rho_b.
\end{align}

We show the $\Gamma$-convergence of
 $$\mathcal A^{\epsilon_n}(\rho,m):=\int_0^1 \frac 1{(1+\epsilon_n \dot W^{\delta})^2} \sum_{ij} \frac {m_{ij}^2}{\theta^{ij}(\rho)}ds, \; \epsilon_n\to0.$$
For a given $(\rho,m)\in C_F(\rho^a,\rho^b)$ and for a sequence $(\rho^{\epsilon_n},m^{\epsilon_n})\in C_F(\rho^a,\rho^b)$ converging to $(\rho,m),$ we have that 
\begin{align*}
\liminf_{n\to \infty}\mathcal A^{\epsilon_n}(\rho^n,m^n)\ge \liminf_{n\to \infty}\int_0^1 \frac 1{(1+\epsilon_n |\dot W^{\delta}|)^2} \sum_{ij} \frac {m_{ij}^2}{\theta^{ij}(\rho)}ds \ge \mathcal A(\rho,m).
\end{align*}
By the dominated convergence theorem, it follows that
\begin{align*}
&\lim_{\epsilon \to 0} \inf_{\rho,\widehat v}[\int_0^1 \frac 12 \frac 1{(1+\epsilon \dot W^{\delta})^2}\<\widehat v_t,\widehat v_t\>_{\theta(\rho_t)} dt]\\
 & \le \inf_{\rho,\widehat v}\limsup_{\epsilon \to 0}\; [\int_0^1 \frac 12 \frac 1{(1+\epsilon \dot W^{\delta})^2}\<\widehat v_t,\widehat v_t\>_{\theta(\rho_t)} dt]\\
&=\inf_{\rho, v}[\int_0^1 \frac 12 \< v_t,v_t\>_{\theta(\rho_t)} dt].
\end{align*} 
Combining the above estimates, we have that the limit of optimal control with common noise \eqref{ot-graph-sp-lim} is the classical optimal control a.s.

\section{Conclusions}
In this paper, using the notion of common noise, we establish the initial value and two-point boundary value problems of stochastic Wasserstein Hamiltonian flows on finite graph. We show the local well-posedness of the initial value problem always holds, up to a positive time, for stochastic Wasserstein Hamiltonian flow and provide a sufficient condition of its global well-posedness. For the boundary value problem, by exploiting the Wong--Zakai approximation, we obtain the existence of the minimizer of optimal control problem perturbed by common noise and derive its dual formula. However, many questions remain to be answered. For example, how to show the existence of the minimizer of optimal control problem driven by the other Wiener process (not common noise) ? Does the minimizer exist for the general variational principle with common noise? 
When considering the lattice graphs, can we get some characterizations of the minimizer for the continuous problem if the mesh size is reduced to zero? Although our focus is on using common noise in this paper, we hope the results may shed light on the investigation of Wasserstein Hamiltonian flow with other types of noise too.

\bibliographystyle{plain}
\bibliography{references}

\end{document}